\UseRawInputEncoding
\documentclass[a4paper, 12pt,reqno]{amsart}
\usepackage{amsmath,amssymb,amsthm,enumerate}
\allowdisplaybreaks
\theoremstyle{plain}
\newtheorem{theorem}{Theorem}
\newtheorem*{theorem*}{Theorem}

\newtheorem*{lemma*}{Lemma}
\theoremstyle{definition}

\newtheorem*{definition*}{Definition}
\theoremstyle{remark}

\newtheorem*{remark*}{Remark}

\newtheorem*{statement*}{Statement}
\newtheorem{corollary}{Corollary}
\newtheorem*{corollary*}{Corollary}

\begin{document}
\title[Minkowski type functions]{Minkowski type functions on probability distributions}

\author{Symon Serbenyuk}



\maketitle
\text{\emph{simon6@ukr.net}}\\
\text{\emph{Kharkiv National University of Internal Affairs,}}\\
\text{\emph{L.~Landau avenue, 27, Kharkiv, 61080, Ukraine}}

\begin{abstract}
In this research, Minkowski type functions which are  constructed on certain probability distributions, are introduced. There are investigated differential, integral, and other properties of these functions.

\end{abstract}

\section{Introduction}

Pathological mathematical objects with complicated local structure, which are previously called ``unexpected" or ``mathematical monsters", have an applied character and interdisciplinary nature due to their complexity and ambiguity of properties in an arbitrary however small neighborhood of an arbitrary point. Complicated local structure of a mathematical object consists in the fact that in any arbitrarily small neighborhood of a point there are certain features for the structure of the object, for example, the transformation of the derivative to zero and infinity, non-differentiability, etc. In real analysis, such investigations are devoted to fractal sets and singular, nowhere monotonic, and non-differentiable functions.

An interest in these functions is explained by their relationships with fractals and fractal multiformalism (the motivation for such research is noted in \cite{ALSW2024, ALSW2024a, AS2021, CLS2024, DS2020, D2021, DS2023, DSM2021, Selmi2021, 8,16,17, Hensley, Hirst, 8, 10, 16}, etc.), as well as by the fact that such functions are an important tool for modeling real objects, processes, and phenomena (in physics, chemistry, biology, economics, technology, etc.). Examples of using or applications of pathological mathematical objects in one or more areas of science can be easily found in a number of sources, including Wikipedia.

Researches on modeling various examples of functions with complicated local structure (including modeling the simplest examples of such functions) was initiated by the classics of world mathematics and continued in the researches of their followers, in particular, various examples of such classes of functions appeared in the scientific results of Bolzano, Weierstrass, Darboux, Dini,
Cantor, Minkowski, Riemann, Salem, and other scientists.

In investigations of classical singular functions, it is worth noting the researches of the properties or construction of generalizations of the Salem function \cite{ACFS2017, S1943, 11, 1, 2, 3, 4, 9, 17, 18, 19, 20}, Cantor function \cite{F1994, D1993, D1995, Zh1991}, and Minkowski function. The last function will be considered below in more detail.

This function was modeled by Minkowski in~\cite{M1904} in 1904. The motivation of this was to matching all quadratic irrationals in [0, 1] to the periodic dyadic rationals. The used techniue is geometric, but in 1938, an analytic expression was given by Denjoy in~\cite{D1938}, where also was proven that this is a singular function, i.e., its derivative is zero almost everywhere in $[0, 1]$.
In~\cite{S1943}, the singularity was proven by using one metric property of continued fractions (i.e., using a certain set, the Lebesgue measure of which is equal to $1$). So, the Minkowski function is a function of the form
$$
?(x)=2^{1-a_1}-2^{1-(a_1+a_2)}+\dots+(-1)^{n-1}2^{1-(a_1+a_2+\dots+a_n)}+\dots ,
$$
where $x\in[0, 1]$, $a_n \in \mathbb N$, and
\begin{equation}
\label{eq:cfrac}
x= a_0 + \cfrac{1}{a_1 + \cfrac{1}{a_2 + \ldots}}.
\end{equation}

In particular, investigations of the properties of the Minkowski function was carried out by R.~Salem, G.~Panti, J.~R.~Kinney, G.~Alkauskas, P.~Viader, J.~Paradis, L.~Bibiloni and other scientists, but a number of questions regarding the properties of this function is remained open for a century. The main reason for this was both the complexity of specifying the function, since its
argument is given in terms of the continued fraction expansion and this expansion have a ``rather complicated non-self-similar geometry", and the singularity. In the future, the attention of scientists only increased to the study of the properties of the Minkowski function. In particular, there are investigated the following: the derivative of $?$ can only take $0$ or $\infty$ (see \cite{PVB2001}, where the conditions which will ensure one value or the other are determined); the Minkowski function does not preserve the Hausdorff dimension (\cite{18}); integral transforms and the moments of the Minkowski question mark function (\cite{A2008, A2010}); certain relationships of properties of the Minkowski function and the multifractal formalism for Stern--Brocot intervals (\cite{KS2008}), etc. More detail surveys are given in \cite{A2008, C2003}. Finally, one can note that the Minkowski function is related to fractals (for example, see \cite{KS2008, 18}).

In the present research, functions of Minkowski type are introduced in terms when an argument represented by certain probability distributions. The main attention is given to properties of modeled functions. These function are useful in approximation theory and discovering relationships between various numeral systems (in our case, between sign-variable binary expansions (on such numeral systems and their generalizations see in \cite{{S2024-NumeralSystems, Symon2023-numeral-systems, S.Serbenyuk2016, S.Serbenyuk 2018}}) and expansions induced probability distributions).

\section{The main object}

This section is devoted to the definition of the Minkowski type function on probability distributions.

 Let us begin with descriptions of expansions of arguments. 

In~\cite{JN2022}, the following expansion of real numbers from $[0,1)$ was introduced:
\begin{equation}
\label{eq}
\widehat{p_{i_1}}+\sum^{\infty} _{k=1}{\left(\widehat{p_{i_{k+1}}}p_{i_1}p_{i_2} \cdots p_{i_k}\right)}:=\pi_P((i_k))=\pi_P(i_1, i_2, \dots, i_k, \dots),
\end{equation} 
where for all positive integers $k$ the following conditions hold: $i_k\in\mathbb N$, $p_{i_k}\in (0,1)$, as well as 
$$
\sum^{\infty} _{i_k=1}{p_{i_k}}=1, ~~~~~~~\widehat{p_{i_k}}=\sum^{i_k-1} _{j=1}{p_{j}}~~~\text{for}~i_k>1,
$$
and $\widehat{p_{1}}=0$, $P:=(p_i)_{i\in\mathbb N}$ is a probability distribution supported by $\mathbb N$.

Let us note auxiliary properties of the considered expansion (\cite[p. 2]{JN2022}):
\begin{enumerate}
\item The map $\pi_P: \mathbb N^{\mathbb N}\to [0,1)$ is a bijection.
\item For a cylinder $\Lambda^\pi _{c_1c_2 \ldots c_n}$, which is a set of elements with fixed $c_j$ for $j=\overline{1,n}$, the following hold:
$$
\Lambda^\pi _{c_1}=\left[\widehat{p_{c_1}}, \widehat{p_{c_1+1}}\right)=\left[\sum^{c_1-1} _{i=1}{p_i}, \sum^{c_1} _{i=1}{p_i}\right)
$$
and
$$
\Lambda^\pi _{c_1c_2 \ldots c_n}=\left[\pi_P(c_1, c_2, \dots, c_{n-1}, c_n, 1, 1, 1, \dots), \pi_P(c_1, c_2, \dots, c_{n-1}, [c_n+1], 1, 1, 1, \dots)\right)
$$
$$
=\left[\widehat{p_{c_1}}+\sum^{n-1} _{k=1}{\left(\widehat{p_{c_{k+1}}}p_{c_1}p_{c_2} \cdots p_{c_k}\right)}, \widehat{p_{c_1}}+\sum^{n-1} _{k=1}{\left(\widehat{p_{c_{k+1}}}p_{c_1}p_{c_2} \cdots p_{c_k}\right)}+p_{c_1}p_{c_2} \cdots p_{c_n}\right)
$$
with the Lebesgue measure, which is equal to
$$
p_{c_1}p_{c_2} \cdots p_{c_n}.
$$
\item For $x=\pi_P(i_1, i_2, \dots, i_k, \dots)$ and  $x_0=\pi_P(j_1, j_2, \dots, j_k, \dots)$ such that $i_m=j_m$ for $m=~\overline{1, u}$, we have
$$
|x-x_0|<\left(\max\{p_i: i\in \mathbb N\}\right)^u.
$$
\end{enumerate}
This expansions was applied for studying of fractals in \cite{JN2022}.

 In the present research, this expansion will be applied for modeling Minkowski function for future investigations of approximations of real numbers, relationships between probability distribution expansions (including cases of irrational probabilities) and sign-variable binary expansions, etc. Really, let us consider the Minkowski type function of the form:
$$
x=\pi_P((i_k)) \stackrel{M_\pi}{\longrightarrow} \sum^{\infty} _{k=1}{\left((-1)^{k-1}2^{1-i_1-i_2-\dots -i_k}\right)}=M_\pi(x)=y.
$$

\section{Properties of functions }

\begin{theorem}
The function $M_\pi$ has the following properties on $[0,1)$:
\begin{itemize}
\item $M_\pi$ is a continuous function.
\item $M_\pi$ is a nowhere monotonic function.
\item $M_\pi$ is a singular function.
\end{itemize}
\end{theorem}
\begin{proof}
Let us prove the first statement. Suppose $x_1, x_2$ are numbers from $[0,1)$ such that
$$
x _1=\pi_P((a_k)):=\pi_P(c_1, c_2, \dots, c_l, a_{l+1}, a_{l+2}, \dots, a_{l+k}, \dots\dots)
$$
and
$$
x _2=\pi_P((b_k)):=\pi_P(c_1, c_2, \dots, c_l, b_{l+1}, b_{l+2}, \dots, b_{l+k}, \dots\dots),
$$
where sequences $(a_n)_{n>l}$ and $(b_n)_{n>l}$ contain one or more than one different numbers. Suppose $a_l\ne b_l$; then 
$$
M_\pi(x_2) -M_\pi(x_1)=(-1)^l 2^{1-(c_1+c_2+\dots +c_l)}\left(B_l-A_l\right),
$$
where
$$
A_l= \frac{1}{2^{a_{l+1}}}-\frac{1}{2^{a_{l+1}+a_{l+2}}}+\dots+\frac{(-1)^{t-1}}{2^{a_{l+1}+a_{l+2}+\dots +a_{l+t}}}+\dots
$$
$$
B_l= \frac{1}{2^{b_{l+1}}}-\frac{1}{2^{b_{l+1}+b_{l+2}}}+\dots+\frac{(-1)^{t-1}}{2^{b_{l+1}+b_{l+2}+\dots +b_{l+t}}}+\dots.
$$
Whence, using a fact that $0<M_\pi<1$ holds, we get
$$
\lim_{x_2-x_1\to 0}{\left|M_\pi(x_2) -M_\pi(x_1)\right|}=\lim_{l\to\infty}{\frac{1}{2^{c_1+c_2+\dots +c_n}}}=\lim_{l\to\infty}{\frac{1}{2^{n}}}=0.
$$
Hence $M_\pi$ is a continuous function at any $x \in [0,1)$, because $x$ has the unique representation.

Let us evaluate the difference $M_\pi(x_2) -M_\pi(x_1)$. For this, for simplifications of notations, use an auxiliary notion of the shift operator $\sigma$. That is, for $x=\pi_P((i_k))$
$$
\sigma(x)= \pi_P((i_k)\setminus\{i_1\})= \pi_P(i_2, i_3, \dots ).
$$
In the other words, 
$$
\sigma(x)=\widehat{p_{i_2}}+\sum^{\infty} _{k=2}{\left(\widehat{p_{i_{k+1}}}p_{i_2}p_{i_3} \cdots p_{i_k}\right)}
$$
and
$$
x=\widehat{p_{i_1}}+p_{i_1}\sigma(x).
$$

By analogy,
$$
\sigma^n(x)=\widehat{p_{i_{n+1}}}+\sum^{\infty} _{k=n+1}{\left(\widehat{p_{i_{k+1}}}p_{i_{n+1}}p_{i_{n+2}} \cdots p_{i_k}\right)}
$$
and
$$
\sigma^n(x)=\widehat{p_{i_{n+1}}}+p_{i_{n+1}}\sigma^{n+1}(x)
$$
for $n=0, 1, 2, 3, \dots$, where $\sigma^n(x)=x$.

So, 
$$
M_\pi(x_2) -M_\pi(x_1)=(-1)^l 2^{1-(c_1+c_2+\dots +c_l)}\left(M_\pi\left(\sigma^{l}(x_2)\right)-M_\pi\left(\sigma^{l}(x_1)\right)\right),
$$
and
$$
-2^{1-(c_1+c_2+\dots +c_l)}<M_\pi(x_2) -M_\pi(x_1)<2^{1-(c_1+c_2+\dots +c_l)}.
$$
Here $\Delta M_\pi=M_\pi(x_2) -M_\pi(x_1)$ can be positive or negative number without depending on the sign of the difference $b_{l+1}-a_{l+1}$. The key parameter is the sign of $(-1)^l$. 

So, our function is a nowhere monotonic function. Really, let us consider one example.

Suppose 
$$
x=\pi_P(v, w, v, w, v, w, v, w, \dots).
$$
Then
$$
M_\pi(x)=2\frac{2^w-1}{2^{v+w}-1}.
$$
If $x_1=\pi_P(1, 2, 1, 2, 1, 2, \dots )$ and $x_2=\pi_P(2, 1, 2, 1, 2, 1,  \dots )$, then $x_1<x_2$ and 
$$
\Delta M_\pi=M_\pi(x_2) -M_\pi(x_1)=\frac 2  7- \frac 6 7<0.
$$
If $x_1=\pi_P(1, 2, 1, 2, 1, 2, \dots )$ and $x_2=\pi_P(1, 4, 1, 4, 1, 4,  \dots )$, then $x_1<x_2$ and 
$$
\Delta M_\pi=M_\pi(x_2) -M_\pi(x_1)=\frac{30}{31}- \frac 6 7>0.
$$

 Let us discuss on the  derivative. Let us consider a cylinder $\Lambda^\pi _{c_1c_2 \ldots c_n}$. Then
$$
M_\pi\left(\sup \Lambda^\pi _{c_1c_2 \ldots c_n}\right)-M_\pi\left(\inf \Lambda^\pi _{c_1c_2 \ldots c_n}\right)
$$
$$
=M_\pi\left(    \pi_P(c_1, c_2, \dots, c_{n-1}, [c_n+1], 1, 1, 1, \dots) \right)-M_\pi\left(    \pi_P(c_1, c_2, \dots, c_{n-1}, c_n, 1, 1, 1, \dots) \right)
$$
$$
=(-1)^{n-1}2^{-(c_1+c_2+\dots +c_{n-1}+c_n)}+(-1)^{n}2^{-(c_1+c_2+\dots +c_{n-1}+c_n)}M_\pi(\pi_P(1, 1, 1, \dots))
$$
$$
-(-1)^{n-1}2^{1-(c_1+c_2+\dots -c_{n-1}+c_n)}+(-1)^{n}2^{1-(c_1+c_2+\dots +c_{n-1}+c_n)}M_\pi(\pi_P(1, 1, 1, \dots))
$$
$$
=(-1)^{n}2^{-(c_1+c_2+\dots +c_{n-1}+c_n)}+(-1)^{n+1}2^{-(c_1+c_2+\dots +c_{n-1}+c_n)}M_\pi(\pi_P(1, 1, 1, \dots))
$$
$$
=\frac{(-1)^{n}}{3\cdot 2^{c_1+c_2+\dots +c_{n-1}+c_n}}
$$
and
$$
\left| \Lambda^\pi _{c_1c_2 \ldots c_n}\right|=p_{c_1}p_{c_2} \cdots p_{c_n}.
$$
Whence,
$$
\lim_{n\to\infty}\frac{\left|M_\pi\left(\sup \Lambda^\pi _{c_1c_2 \ldots c_n}\right)-M_\pi\left(\inf \Lambda^\pi _{c_1c_2 \ldots c_n}\right)\right|}{\left| \Lambda^\pi _{c_1c_2 \ldots c_n}\right|}=\lim_{n\to\infty}{\frac{1}{3\cdot 2^{c_1+c_2+\dots +c_{n-1}+c_n}p_{c_1}p_{c_2} \cdots p_{c_n}}}.
$$
Using
$$
\varrho_n:=\frac{\left|M_\pi\left(\sup \Lambda^\pi _{c_1c_2 \ldots c_n}\right)-M_\pi\left(\inf \Lambda^\pi _{c_1c_2 \ldots c_n}\right)\right|}{\left| \Lambda^\pi _{c_1c_2 \ldots c_n}\right|},
$$
$$
\frac{\varrho_n}{\varrho_{n-1}}=\frac{1}{p_{c_n}2^{c_n}},
$$
and the fact that $p_1, p_2, \dots $ are fixed nunbers from $(0, 1)$, as well as that $c_n\in \mathbb N$ and that our function is continuous, we obtain $M_\pi$ is singular.
\end{proof}

\begin{theorem}  A   system of functional equations of the form
\begin{equation}
\label{eq:system}
{h\left(\sigma^{n-1}(x)\right)}=\frac{1}{2^{i_{n}}}\left(1-h\left(\sigma^n(x)\right)\right),
\end{equation}
where $x$ represented by expansion \eqref{eq} induced by probability distribution on $\mathbb N$, $i_n\in \mathbb N$, $n=1,2, \dots$, $\sigma$ is the shift operator, and $\sigma_0(x)=x$, has the unique solution
$$
\frac{M_\pi(x)}{2}=2^{-i_1}-2^{-i_1-i_2}+\dots + (-1)^{n-1}2^{-i_1-i_2-\dots -i_n}+\dots
$$
in the class of determined and bounded on $[0, 1)$ functions.
\end{theorem}
\begin{proof} Since the function  is a determined and bounded  function on $[0,1)$, according to system~\eqref{eq:system},  we obtain
$$
\frac{M_\pi(x)}{2}=\frac{1}{2^{i_{1}}}\left(1- h(\sigma(x))\right)=\frac{1}{2^{i_{1}}}- \frac{1}{2^{i_{1}}}\left(\frac{1}{2^{i_{2}}}\left(1- h(\sigma^2(x))\right)\right)
$$
$$
\dots =2^{-i_1}-2^{-i_1-i_2}+\dots + (-1)^{n-1}2^{-i_1-i_2-\dots -i_n}+\frac{(-1)^{n}}{2^{i_1+i_2+\dots +i_n}}h(\sigma^n(x)).
$$
This statement is true, because
$$
\lim_{n\to\infty}{\frac{1}{2^{i_1+i_2+\dots +i_n}}h(\sigma^n(x))}=0.
$$
\end{proof}

From the last statement it follows the following.
\begin{corollary}
Suppose for $t=1, 2, 3 \dots$,
$$
\psi_t: \left\{
\begin{array}{rcl}
x^{'}&=&\widehat{p_t}+p_tx\\
y^{'} & = &\frac{1}{2^t}(1-y)\\
\end{array}
\right.
$$
are affine transformations. Then the graph $\Gamma$ of  $M_\pi$ is a self-affine set of $\mathbb R^2$, as well as
$$
\Gamma= \bigcup^{\infty} _{t=1}{\psi_t(\Gamma)}.
$$
\end{corollary}

\begin{theorem}
For the Lebesgue integral, the following equality holds:
$$
\int^1 _0 {M_\pi(x)dx}=2\frac{\sum^{\infty} _{j=1}{\frac{p_j}{2^j}}}{1+\sum^{\infty} _{a=1}{\frac{p^2 _j}{2^j}}}.
$$
\end{theorem}
\begin{proof} Since
$$
\sigma^n(x)=\widehat{p_{i_{n+1}}}+p_{i_{n+1}}\sigma^{n+1}(x)
$$
and
$$
d(\sigma^n(x))=p_{i_{n+1}}d(\sigma^{n+1}(x))
$$
hold for all $n=0, 1, 2, 3, \dots$, using system \eqref{eq:system},  i.e., 
$$
M_\pi(\sigma^n(x))=\frac{1}{2^{i_{n}}}\left(1-M_\pi(\sigma^{n+1}(x))\right)
$$
according to self-affine properties of $?_{EM}$, we obtain

$$
\frac I 2=\frac 1 2 \int^1 _0 {M_\pi(x)dx}=\int^1 _0{ \frac{1}{2^{i_{1}}}\left(1-M_\pi(\sigma(x))\right)}dx
$$
$$
=\int^1 _0 { \frac{1}{2^{i_{1}}}}dx-\int^1 _0 { \frac{p_{i_1}}{2^{i_{1}}}M_\pi(\sigma(x))}d(\sigma(x))
$$
$$
=\sum^{\infty} _{j=1}{\int^{\widehat{p_{j+1}}} _{\widehat{p_{j}}}{ \frac{1}{2^{j}}}}dx-\sum^{\infty} _{j=1}{\int^{\widehat{p_{j+1}}} _{\widehat{p_{j}}}{ \frac{p_j}{2^{j}}M_\pi(\sigma(x))}d(\sigma(x))}
$$
$$
=\sum^{\infty} _{j=1}{\frac{p_j}{2^j}}-\left(\sum^{\infty} _{j=1}{\frac{p^2 _j}{2^j}}\right)\int^1 _0{M_\pi(\sigma(x))d(\sigma(x))}.
$$
Using the notations
$$
\alpha:= \sum^{\infty} _{j=1}{\frac{p_j}{2^j}}~~~\text{and}~~~\gamma:=\sum^{\infty} _{j=1}{\frac{p^2 _j}{2^j}},
$$
we have
$$
\frac I 2=\alpha-\gamma\left(\int^1 _0{ \frac{1}{2^{i_{2}}}\left(1-M_\pi(\sigma^2(x))\right)}d(\sigma(x))\right)
$$
$$
=\alpha-\gamma\left(\int^1 _0 { \frac{1}{2^{i_{2}}}}d(\sigma(x))-\int^1 _0 { \frac{p_{i_2}}{2^{i_{2}}}M_\pi(\sigma^2(x))}d(\sigma^2(x))\right)
$$
$$
=\alpha-\gamma\left(\alpha-\gamma\int^1 _0{M_\pi(\sigma^2(x))d(\sigma^2(x))}\right)=\alpha-\alpha\gamma+\gamma^2\int^1 _0{M_\pi(\sigma^2(x))d(\sigma^2(x))}=\dots
$$
$$
\dots=\alpha-\alpha\gamma+\gamma^2\alpha-\gamma^3\alpha+\dots+(-1)^{n-1}\alpha\gamma^{n-1}+(-1)^{n}\gamma^n\int^1 _0{M_\pi(\sigma^n(x))d(\sigma^n(x))}.
$$

Since $0<p_j<1$ and
$$
\gamma=\sum^{\infty} _{j=1}{\frac{p^2 _j}{2^j}}<\sum^{\infty} _{j=1}{\frac{1}{2^j}}=1,
$$
then $\gamma_n \to 0 (n \to\infty)$. That is, 
$$
\frac I 2=\sum^{\infty} _{n=1}{(-1)^{n-1}\alpha\gamma^{n-1}}. 
$$
\end{proof}

\section*{Statements and Declarations}
\begin{center}

{\bf{Competing Interests}}

\emph{The author states that there is no conflict of interest}

{\bf{Information regarding sources of funding}}

\emph{No funding was received}

{\bf{Data availability statement}}

\emph{The manuscript has no  associated data}

{\bf{Study-specific approval by the appropriate ethics committee for research involving humans and/or animals, informed consent if the research involved human participants, and a statement on welfare of animals if the research involved animals (as appropriate)}}

\emph{There are not suitable  for this research}

\end{center}


\begin{thebibliography}{9}
\bibitem{ALSW2024}
R. Achour, Zh. Li, B. Selmi, and T. Wang. 2024. A multifractal formalism for new general fractal
measures. \emph{Chaos, Solitons \&amp; Fractals} 181, 114655.
https://doi.org/10.1016/j.chaos.2024.114655.

\bibitem{ALSW2024a}
R. Achour, Zh. Li, B. Selmi, and T. Wang. 2024. General fractal dimensions of graphs of products
and sums of continuous functions and their decompositions. \emph{Journal of Mathematical
Analysis and Applications} 538 (2), 128400. https://doi.org/10.1016/j.jmaa.2024.128400.

\bibitem{A2008}
Alkauskas, G. (2008) Integral transforms of the Minkowski question mark function. PhD thesis,
University of Nottingham. https://eprints.nottingham.ac.uk/10641/1/alkauskas\_thesis.pdf

\bibitem{A2010}
Alkauskas, G. (2010). The moments of minkowski question mark function: the dyadic period
function. \emph{Glasgow Mathematical Journal} 52(1), 41---64. doi:10.1017/S0017089509990152

\bibitem{ACFS2017}
de Amo E. , Carrillo M. D., Fern\'andez-S\'anchez J. A Salem generalized function. \emph{Acta
Math. Hungar. } 2017. Vol.~{151} , no.~2. Pp. 361--378. https://doi.org/10.1007/s10474-017-0690-
x

\bibitem{AS2021}
Attia, N., and B. Selmi. 2021. A multifractal formalism for Hewitt-Stromberg measures.
\emph{Journal of Geometric Analysis} 31, 825–862.

\bibitem{PVB2001}
J. Paradis, P. Viader, L. Bibiloni. (2001). The derivative of Minkowski's $? (x)$ function. \emph{J.
Math. Anal. Appl.}, 253, 107-125.

\bibitem{M1904}
H. Minkowski. Verhandlungen des III internationalen mathematiker-kongresses in Heidelberg,
1904. Also in ``Gesammelte Abhandlungen, 1991," Vol. 2, pp. 50--51 for the $?$ function.

\bibitem{D1938}
A. Denjoy. (1938). Sur une fonction reelle de Minkowski, \emph{J. Math. Pures Appl. } 17, 105--
151.

\bibitem{S1943}
R. Salem. (1943). On some singular monotonic functions which are strictly increasing.
\emph{Trans. Amer. Math. Soc.} 53, 427---439.

\bibitem{CLS2024}
D. Cheng, Zh. Li, and B. Selmi. (2024). On the general fractal dimensions of hyperspace of compact
sets. Fuzzy Sets and Systems 488, 108998. https://doi.org/10.1016/j.fss.2024.108998.

\bibitem{C2003}
R. M. Conley. A Survey of the Minkowski $?(x)$ Function, M.S. thesis, West Virginia University,
Morgantown, WV, 2003.

\bibitem{D1993}
R. Darst. The Hausdorff dimension of the nondifferentiability set of the Cantor function is
$[ln(2)/ln(3)]^2$.
\emph{Proc. Amer. Math. Soc.} 119 (1993), no. 1, 105---108.

\bibitem{D1995}
R. Darst. (1995). Hausdorff dimension of sets of non-differentiability points of Cantor functions.
\emph{Mathematical Proceedings of the Cambridge Philosophical Society} 117(1), 185---191.
doi:10.1017/S0305004100073011

\bibitem{DS2020}
Douzi, Z., and B. Selmi. (2020). On the mutual singularity of multifractal measures.
\emph{Electronic Research Archive} 28, 423---432.

\bibitem{D2021}
Douzi, Z., et al. (2021). Another example of the mutual singularity of multifractal measures.
\emph{Proyecciones} 40, 17---33.

\bibitem{DS2023}
Douzi, Z., and B. Selmi. (2023). On the mutual singularity of Hewitt-Stromberg measures for which
the multifractal functions do not necessarily coincide. \emph{Ricerche di Matematica}.
https://doi.org/10.1007/s11587-021-00572-6.

\bibitem{DSM2021}
Douzi, Z., B. Selmi, and A.B. Mabrouk. (2021). The refined multifractal formalism of some
homogeneous Moran measures. \emph{The European Physical Journal Special Topics}, 230, 3815--
-3834. https://doi.org/10.1140/epjs/s11734-021-00318-3.

\bibitem{F1994}
J. F. Fleron. (1994). A Note on the History of the Cantor Set and Cantor Function.
\emph{Mathematics Magazine}, 67(2), 136---140.
https://doi.org/10.1080/0025570X.1994.11996201

\bibitem{Hensley}
Hensley D. Continued fraction Cantor sets, Hausdorff dimension and functional analysis.
\emph{Journal of Number Theory}. 1992. Vol. 40. Pp. 336--358. https://doi.org/10.1016/0022-
314X(92)90006-B

\bibitem{Hirst}
Hirst K. E. Fractional dimension theory of continued fractions. \emph{The Quarterly Journal of
Mathematics}. 1970. Vol. 21, no. 1. Pp. 29--35. https://doi.org/10.1093/qmath/21.1.29

\bibitem{KS2008}
M. Kesseb\"ohmer and B. O. Stratmann. (2008). Fractal analysis for sets of non-differentiability of
Minkowski's question mark function. \emph{Journal of Number Theory} 128, No. 9, 2663-2686.
https://doi.org/10.1016/j.jnt.2007.12.010.

\bibitem{Minkowsky1905}
Minkowski H. Zur Geometrie der Zahlen.\emph{Verh. d. 3. intern. Math.-Kongr. Heidelb}. 1905.
Pp. 164--173. (German)

\bibitem{Minkowsky1911}
Minkowski H. Zur Geometrie der Zahlen. In: Minkowski, H. (ed.) Gesammeine Abhandlungen,
Band 2,
pp. 50--51. Druck und Verlag von B. G. Teubner, Leipzig und Berlin, 1911.

\bibitem{JN2022}
Neunh\"auserer, J. 2022. Representations of real numbers induced by probability distributions on $\mathbb N$. \emph{Tatra Mountains Mathematical Publications} 82, 1--8. https://doi.org/10.2478/tmmp-2022-0014

\bibitem{S1943}{ Salem R.} {On some singular monotonic functions which are stricly increasing.}
\emph{ Trans. Amer. Math. Soc.}. 1943. Vol. {53}. Pp. 423--439.

\bibitem{Selmi2021}
Selmi, B. (2021). The mutual singularity of multifractal measures for some non-regularity Moran
fractals. \emph{Bulletin of the Polish Academy Sciences Mathematics} 69, 21---35.

\bibitem{11}
S. Serbenyuk. 2024. Relationships between singular expansions of real numbers. \emph{ The
Journal of Analysis.}
32(6), 3655---3675. https://doi.org/10.1007/s41478-024-00825-1

\bibitem{1}  Serbenyuk, S. 2024. Singular modifications of a classical function. \emph{ Acta
Mathematica Hungarica.} 172, no. 1. Pp. 206--222. https://doi.org/10.1007/s10474-024-01406-1

\bibitem{2}  Serbenyuk, S. 2024. The generalized Salem functions defined in terms of certain Cantor
expansions. \emph{ The Journal of Analysis.} 32(3), 1645---1660. https://doi.org/10.1007/s41478-
023-00711-2.

\bibitem{3} S. Serbenyuk. 2023. A certain modification of classical singular function, \emph{
Bolet\'in de la Sociedad Matem\'atica Mexicana}. 29, no.3. Article Number 88.
https://doi.org/10.1007/s40590-023-00569-1

\bibitem{4} Serbenyuk, S. 2024. Functional equations, alternating expansions, and generalizations of
the Salem functions. \emph{Aequationes Mathematicae}. 98(5), 1211---1223.
https://doi.org/10.1007/s00010-023-00992-9

\bibitem{S2024-NumeralSystems}
Serbenyuk, S.  2024. One example of singular representations of real numbers from the unit interval. \emph{Proceedings of the Institute of Mathematics and Mechanics} \textbf{50}, no. 1,  96---103. https://doi.org/10.30546/2409-4994.2024.50.1.96

\bibitem{Symon2023-numeral-systems}  Serbenyuk, S. 2023. Some types of numeral systems and their modeling,\emph {The Journal of Analysis } { \textbf{31}}, 149--177. https://doi.org/10.1007/s41478-022-00436-8


\bibitem{8}  Serbenyuk, S. 2022. Some fractal properties of sets having the Moran structure.
\emph{Tatra Mountains Mathematical Publications}. 81, no. 1. Pp. 1--38.
https://doi.org/10.2478/tmmp-2022-0001

\bibitem{9} S. Serbenyuk. 2021. Systems of functional equations and generalizations of certain
functions. \emph{Aequationes Mathematicae}. 95, no. 5. Pp. 801--820,
https://doi.org/10.1007/s00010-021-00840-8

\bibitem{10}  Serbenyuk, S. (2021). Certain singular distributions and fractals. \emph{Tatra
Mountains Mathematical Publications}. 79, no. 2. Pp. 163--198. https://doi.org/10.2478/tmmp-2021-
0026

\bibitem{16}Serbenyuk,  S. O. (2020). One distribution function on the Moran sets. \emph{Azerbaijan Journal of Mathematics}. 10, no.2. Pp. 12--30.
https://www.azjm.org/volumes/1002/pdf/1002-2.pdf, arXiv:1808.00395.



\bibitem{S.Serbenyuk 2018} Serbenyuk, S.  2020. Generalizations of certain representations of real numbers, {Tatra Mountains Mathematical Publications} \textbf{77}, 59--72. https://doi.org/10.2478/tmmp-2020-0033, arXiv:1801.10540

\bibitem{17} Serbenyuk, S. (2019). On one application of infinite systems of functional equations in
function theory. \emph{ Tatra Mountains Mathematical Publications}. Vol. {74}. Pp. 117--144.
https://doi.org/10.2478/tmmp-2019-0024

\bibitem{18} Serbenyuk, S. (2018). On one fractal property of the Minkowski function.
\emph{Revista de la Real Academia de Ciencias Exactas, F\'isicas y Naturales. Serie A. Matem\'aticas}. Vol. {112}, no.~2. Pp. 555--559. DOI:10.1007/s13398-017-0396-5

\bibitem{19} { Serbenyuk, S.O.} (2018).
Non-Differentiable functions defined in~terms of~classical representations of~real numbers.
\emph{Journal of Mathematical Physics, Analysis, Geometry}. 14, no.~2. Pp.
197--213. https://doi.org/10.15407/mag14.02.197




\bibitem{20} { Serbenyuk, S.~O.} (2017). Continuous functions with complicated local structure
defined in terms of alternating Cantor series representation of numbers.\emph{ Journal of
Mathematical Physics, Analysis, Geometry}. {13}, no. 1. Pp. 57--81.
https://doi.org/10.15407/mag13.01.057

\bibitem{30} \emph{ Serbenyuk, S.} (2016). On one class of functions with complicated local
structure. \emph{\v{S}iauliai Mathematical Seminar.} Vol. {11 (19)}. Pp.~75--88. URL:
https://www.researchgate.net/publication/301873839



\bibitem{S.Serbenyuk2016} {Serbenyuk,  S. }2016.  On some generalizations of real numbers representations, arXiv:1602.07929v1 (preprint; in Ukrainian)



\bibitem{Zh1991}
 Zheng, W. (1991). On p-adic cantor function. In: Cheng, MT., Deng, DG., Zhou, XW. (eds)
Harmonic Analysis. Lecture Notes in Mathematics, vol 1494. Springer, Berlin, Heidelberg.
https://doi.org/10.1007/BFb0087776

\end{thebibliography}
\end{document}